\documentclass[12pt]{article}
\usepackage{wrapfig}

\usepackage[top=3cm,left=3cm,right=3cm,bottom=3cm]{geometry}
\usepackage{hyperref}
\usepackage{amssymb,latexsym,amsmath,epsfig,amsthm, tikz, verbatim, hyperref,color,amsthm} 
\usepackage[english]{babel}
\usepackage{tikz-cd} 
\usepackage[section]{placeins}
\makeatletter
 
 \usepackage{framed}
\renewcommand\section{\@startsection {section}{1}{\z@}
{-30pt \@plus -1ex \@minus -.2ex}
{2.3ex \@plus.2ex}
{\normalfont\normalsize\bfseries}}

\renewcommand\subsection{\@startsection{subsection}{2}{\z@}
{-3.25ex\@plus -1ex \@minus -.2ex}
{1.5ex \@plus .2ex}
{\normalfont\normalsize\bfseries}}

\renewcommand{\@seccntformat}[1]{\csname the#1\endcsname. }

\newcommand{\defcolor}[1]{{\color{blue}#1}}
\newcommand{\demph}[1]{\textbf{\defcolor{{\sl #1}}}} 

\makeatother

\newtheorem{theorem}{Theorem}
\newtheorem{Summary}[theorem]{Summary}
\numberwithin{theorem}{section}
\newtheorem{proposition}[theorem]{Proposition}

\newtheorem{corollary}[theorem]{Corollary}

\theoremstyle{definition}
\newtheorem{definition}[theorem]{Definition}
\newtheorem{remark}[theorem]{Remark}
\newtheorem{example}[theorem]{Example}

\begin{document}

\title{The critical curvature degree of an algebraic variety}

\author{Emil Horobe\c{t}}
\date{}

\maketitle
\begin{abstract}
In this article we study the complexity involved in the computation of the reach in arbitrary dimension and in particular the computation of the critical spherical curvature points of an arbitrary algebraic variety. We present properties of the critical spherical curvature points as well as an algorithm for computing them.
\end{abstract}

\section{Introduction}

Sampling an object is an important problem when trying to recover information about its structure. Using geometric information such as the bottlenecks and the local reach in article \cite{DiR20} the authors provide bounds
on the density of the sample needed in order to guarantee that the homology of the variety can
be recovered from the sample. If the sample is finer than the reach, then the homology can be recovered.

The reciprocal of the reach, called the condition number of the underlying system of
polynomials, is also a very important indicator, see for instance \cite{Burg1,Burg2}. Here the authors
derive a complexity analysis of an algorithm for computing homology, using condition numbers.

If our sample data is driven by a geometric model, say an algebraic
variety $X$, then the reach can be computed as a minimum of two quantities (see \cite[Lemma $3.3$]{Aamari}), more precisely, one of which is the radius of the narrowest bottleneck of $X$, and the other one is the minimal radius of spherical curvature on $X$ (see the definition of the reach in \cite[Proposition $6.1$]{NSW} in terms of the radius of the osculating sphere and see \cite[Section $4.2$]{Aamari} for a detailed explanation). From now on we will call the spherical curvature simply curvature.

We deal with the complexity involved in the computation of the reach and in particular the computation of locally minimal, maximal, etc. curvature points. An analysis of the bottleneck degree of a variety was done in \cite{DiRocco,Eklund}, but understanding critical curvature points in arbitrary dimension was open till now. A direct method to find the variety of critical radii of curvature for plane curves is presented in \cite[Theorem $4.8$]{MadMad}. 

What the present article intends to do is a different approach. In this article we want to analyse the set of all points on the generalized evolute (or focal locus or Euclidean Distance Discriminant) to a real variety $X$ that correspond to centres of curvature at points with critical curvature (locally minimal, maximal, saddle type, etc.). We construct a variety containing these points and we will call its degree the \demph{Critical Curvature Degree} of $X$. Our findings can be summarized as follows. 

\begin{Summary}	
	For an arbitrary variety $X$ of any degree in any dimension we
	\begin{itemize}
		\item provide an algorithm to compute the critical curvature degree of $X$ (see algorithm in Example~\ref{ellipse}). This degree describes the complexity of finding all the points of critical curvature on the variety (using for example numerical homotopy methods, see for instance \cite{Brei}); 
		\item give examples to show that this degree also describes the "curviness" of $X$, in the sense that the critical curvature degree of a linear space is $0$ (its ED discriminant is empty), of a circle is $2$ (see Example~\ref{circle}), of an ellipse is $4$ (see Example~\ref{ellipse}) and so on;
		\item show that the set of critical curvature points can be zero dimensional (see Example~\ref{ellipse}) or higher (see Example~\ref{ellipsoid});
		\item finally prove that all singular points of the ED discriminant correspond to points in the critical curvature pairs variety (see Theorem~\ref{Main}). This generalizes classical findings (see \cite{Li19, Port}).
	\end{itemize}
\end{Summary}

\section{The offset discriminant}
To start we need to define the radius- and centre of curvature of a variety at a given regular (not singular) point. In order to do so we adapt an approach via level curves and evolutes or more generally via offset hypersurfaces and offset discriminants. In this section we recall the notion and properties of offset hypersurfaces from \cite{HW19}.

Let us consider $X_{\mathbb{R}}\subset \mathbb{R}^n$ a proper real variety and $\epsilon$ a fixed a real number. In order to use techniques form algebraic geometry we will consider the complexification of $X_{\mathbb{R}}$, that is the complex variety $X_{\mathbb{C}}$ in $\mathbb{C}^n$ defined by the same ideal as $X_{\mathbb{R}}$ and we also consider $\epsilon$ to be a complex number. Given two vectors $x,y\in \mathbb{C}^n$ we define the complexified squared Euclidean distance function $d: \mathbb{C}^n\times \mathbb{C}^n\to\mathbb{C}$, such that $\displaystyle{d(x,y)=\sum_{i=1}^n(x_i-y_i)^2}$. This is a complex valued function, which for real entries provides the real squared Euclidean distance of the two real entries. For any given $y_0\in \mathbb{C}^n$ we define the \textit{$\epsilon$-ball centered around $y_0\in \mathbb{C}^n$} to be $B_{\epsilon}(y_0)=\{x\in \mathbb{C}^n, \text{ s.t. } d(x,y_0)=\epsilon\}$.

The \textit{$\epsilon$-offset hypersurface} is defined to the closure of the union of the centres of $\epsilon$-balls that intersect the variety $X$ non-transversally at some regular (non singular) point $x\in X_{reg}$. For a fixed $\epsilon$ we denote the $\epsilon$-offset hypersurface by $\mathcal{O}_{\epsilon}(X)$.

\begin{figure}[h]
	\begin{center}
		\vskip -0.2cm
		\includegraphics[scale=1.1]{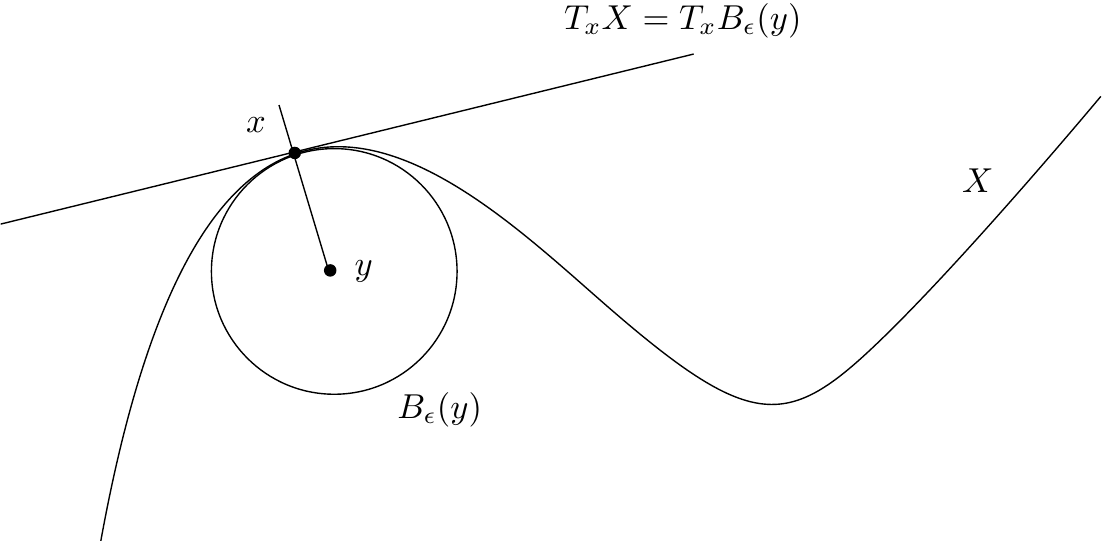}
		\vskip -0.2cm
		\caption{Non-transversal intersection of the variety with the $\epsilon$-ball.}
		\label{Offset hypersurface construction}
	\end{center}
\end{figure}
So if $y\in\mathcal{O}_{\epsilon}(X)$, then there exists an $x\in X_{reg}$, that is a regular point of the variety (non-singular point), such that their squared distance $d(x,y)$ is exactly $\epsilon$ and by the non-transversality $T_x X\subseteq T_x B_{\epsilon}(y)$, that is $x-y\perp T_x X$, where $T_x X$ is the tangent space at $x$ to $X$ and $T_x B_{\epsilon}(y)$ is the tangent space at $x$ to $B_{\epsilon}(y)$.
Now  we consider the closure of the set of all pairs $(x,y)\in \mathbb{C}^n\times \mathbb{C}^n$ such that $d(x,y)=\epsilon$, $x\in X_{reg}$ and $x-y\perp T_x X$. The name of this variety is the \textit{offset correspondence} of $X$ and we denote it by $\mathcal{OC}_{\epsilon}(X)$. This correspondence is a variety in $\mathbb{C}^n\times \mathbb{C}^n$ and is equal to the closure of the intersection
\[\label{offset_equ}
\{(x,y)\in \mathbb{C}^n\times \mathbb{C}^n,\text{ s.t. }x\in X_{reg},\ x-y\perp T_x X\}\cap \{(x,y)\in \mathbb{C}^n\times \mathbb{C}^n,\text{ s.t. } d(x,y)=\epsilon\}.
\]
Observe that the first variety is equal to the Euclidean Distance Degree correspondence, $\mathcal{E}(X)$. This correspondence contains pairs of ``data points" $y\in \mathbb{C}^n$ and corresponding points on the variety $x\in X_{reg}$, such that $x$ is a constrained critical point of the Euclidean distance function $d_{y}(x)=d(x,y)$ (distance from $y$) with respect to the constraint that $x\in X_{reg}$. For more details on this problem we direct the reader to \cite[Section 2]{DHOST16}.
Using the terminology of the Euclidean distance degree problem, we have
\[
\mathcal{OC}_{\epsilon}(X)=\mathcal{E}(X)\cap \{(x,y)\in \mathbb{C}^n\times \mathbb{C}^n,\text{ s.t. } d(x,y)=\epsilon\}.
\]
We have the natural projections $\mathrm{pr}_1: \mathcal{OC}_{\epsilon}(X)\subseteq \mathbb{C}^n\times \mathbb{C}^n \to \mathbb{C}^n$, to the first $n$-tuple of coordinates and $\mathrm{pr}_2: \mathcal{OC}_{\epsilon}(X)\subseteq \mathbb{C}^n\times \mathbb{C}^n \to  \mathbb{C}^n$, to the second $n$-tuple of coordinates. The closure of the image of the first projection is the variety $X$ and the closure of the image of the second projection is the \textit{offset hypersurface} $\mathcal{O}_{\epsilon}(X)$.
\begin{equation}\label{proj}
\begin{tikzcd}
&\mathcal{OC}_{\epsilon}(X) \arrow{ld}[swap]{\mathrm{pr}_{1}}\arrow{rd}{\mathrm{pr}_2}& \\
X\subseteq \mathbb{C}^n & & \mathcal{O}_{\epsilon}(X)\subseteq \mathbb{C}^n
\end{tikzcd}
\end{equation}

We remind the reader that (see \cite[Theorem 4.1]{DHOST16}) the Euclidean Distance correspondence $\mathcal{E}(X)$ is an irreducible variety of dimension $n$ inside $\mathbb{C}^n \times \mathbb{C}^n$.
The first projection extended to $\mathcal{E}(X)$, that is $\mathrm{pr}_1 : \mathcal{E}(X) \to X \subseteq \mathbb{C}_x^n$, is an affine vector bundle over $X_{reg}$, of rank equal to the codimension of $X$
and over generic points the second projection extended to $\mathcal{E}(X)$, that is $\mathrm{pr}_2: \mathcal{E}(X) \to
\mathbb{C}^n$, has finite fibers of cardinality equal to the \textit{Euclidean Distance Degree} (ED degree) of $X$.

The second projection, $\mathrm{pr}_2$, has a branch locus, which will play an important role in our future investigation. This ramification locus is generically a hypersurface in $\mathbb{C}^n$, by the Nagata-Zariski Purity Theorem \cite{Purity2},\cite{Purity1}. The \textit{Euclidean Distance discriminant (ED discriminant)} is the closure of the image of the ramification locus of $\mathrm{pr}_2$, (i.e. the points where the derivative of $\mathrm{pr}_2$ is not of full rank, under the projection $\mathrm{pr}_2$). As in~\cite[Section 7]{DHOST16}, we denote the ED discriminant of the variety $X$ by $\Sigma(X)$.

The offset correspondence is the intersection of the Euclidean Distance correspondence with the hypersurface $\{(x,y)\in \mathbb{C}^n\times \mathbb{C}^n,\text{ s.t. } d(x,y)=\epsilon\}$ in $\mathbb{C}^n\times \mathbb{C}^n$. This intersection is $n-1$ dimensional because $\mathcal{E}(X)$ is not a  subvariety of the latter one (because not all pairs $(x,y)\in \mathcal{E}(X)$ are at $\epsilon$ squared distance from each other). So the offset correspondence, $\mathcal{OC}_{\epsilon}(X)$, is an $n-1$ dimensional variety in $\mathbb{C}^n\times \mathbb{C}^n$. But over a generic point projection $\mathrm{pr}_2$ has finite fibers (of cardinality equal to the ED degree of $X$ by definition), so $\mathcal{O}_{\epsilon}(X)$ is $n-1$ dimensional as well, hence the name offset \textit{hypersurface}. 

Let us suppose that the radical ideal of the offset hypersurface is generated by the polynomial $f_{\epsilon}(y)\in\mathbb{C}[y_1,\ldots,y_n]$, with parameter $\epsilon$. This polynomial is also called the \textit{ED polynomial} in \cite{OS20} and it is the parametric polynomial defining the family of all offset hypersurfaces of $X$ parametrized by the variable $\epsilon$ representing the squared distance. One can consider this as a polynomial in the $\epsilon$ variable as well, so as $f_{\epsilon}(y)\in\mathbb{C}[y_1,\ldots,y_n][\epsilon]$.  Now by computing its $\epsilon$-discriminant one gets $\mathrm{Disc}_{\epsilon}(f)(y)\in\mathbb{C}[y_1,\ldots,y_n]$ which by the definition of the discriminant is a scalar multiple of the resultant of the ED polynomial and its $\epsilon$-derivative, hence it is the defining polynomial of the \textit{envelope of the family of all offset hypersurfaces}. 
This envelope is denoted by $\Delta(X)$ and it is called the \textit{offset discriminant of $X$} (see \cite{HW19}). So we have that
\[\Delta(X)=\{y\in \mathbb{C}^n, \text{ s.t. }\mathrm{Disc}_{\epsilon}(f)(y)=0\}.\] 

This is a hypersurface in $\mathbb{C}^n$ that is tangent to each member of the family of all offset hypersurfaces at some point, and these points of tangency together form the whole envelope. Classically, a point on the envelope can be thought of as the intersection of two infinitesimally adjacent hypersurfaces in the family, meaning the limit of intersections of nearby offset hypersurfaces.

So if $y_0\in\Delta(X)$, then this means that for some $\epsilon$ we have that \[\displaystyle{y_0\in \lim_{\delta\to\epsilon}\mathcal{O}_{\delta}(X)\cap\mathcal{O}_{\epsilon}(X).}\]
This means that there exist a sequence $y_{\delta}\to y_0$, as $\delta\to \epsilon$, such that $y_{\delta}\in \mathcal{O}_{\delta}(X)\cap\mathcal{O}_{\epsilon}(X)$. This means that for every $\delta$ there exists $x_{\delta\epsilon}^1, x_{\delta\epsilon}^2\in X$ such that $d(y_{\delta}, x^1_{\delta \epsilon})=\delta$, $d(y_{\delta}, x^2_{\delta\epsilon})=\epsilon$ and $x_{\delta\epsilon}^1-y_{\delta}\perp T_{x_{\delta\epsilon}^1}X$ and  $x_{\delta\epsilon}^2-y_{\delta}\perp T_{x_{\delta\epsilon}^2}X$. Meaning that $y_{\delta}$ is at the intersection of the two normal lines of $X$, spanned by $x_{\delta\epsilon}^1-y_{\delta}$ and by $x_{\delta\epsilon}^2-y_{\delta}$. 
As $\delta$ goes to $\epsilon$ the sequences $x_{\delta\epsilon}^1$ and $x_{\delta\epsilon}^2$ converge. This is true because $X$ is (algebraically) closed, moreover the limits can not run away to infinity because their distance to $y_0$ is fixed, so we have boundedness as well and hence compactness (and sequential compactness) is fulfilled. And for the limits there are two distinct cases. 

One, if $x_{\delta\epsilon}^1 \to x_{1}$ and $x_{\delta\epsilon}^2 \to x_{2}$, with $x_1\neq x_2$, then $y_0$ is at the intersection of two distinct normal lines with $d(y_0,x_1)=d(y_0,x_2)$, hence being an element of the \textit{bisector hypersurface to $X$} (see for instance \cite{HW19}[Section~$2.3$]) or in other words an element of the symmetry set of $X$ or the union of all Voronoi boundaries, see for instance \cite{Voronoi}. Let us denote the bisector hypersurface by $B(X,X)$.

Two, if $x_{\delta\epsilon}^1, x_{\delta\epsilon}^2 \to x_{0}$, then in the limit $y_0$ will be at the intersection of two "infinitesimally close" normal lines to $X$, and it can be defined to be \demph{the center of curvature} at $x_0$ to $X$ (generalizing the classical definition by Cauchy~\cite{Cau}). 
Of course the closure of this set is also called the ED discriminant and it is denoted by $\Sigma(X)$.
For plane curves it is classical of course that the locus of centres of curvature comprise the evolute of the curve. The above definition generalizes this.

To summarize these two cases we formulate the proposition below, which can be also found in \cite[Proposition 2.5.]{OS20} and in \cite[Proposition 2.13]{HW19}) as well.

\begin{proposition}\label{Split}
Let $X$ be an algebraic variety, then the offset discriminant decomposes into the bisector hypersurface and the ED discriminant of the variety. So we have that\[\Delta(X)=B(X,X)\cup\Sigma(X).\]
\end{proposition}

\begin{remark}\label{finproj}
If one picks a generic point $y_0\in \Sigma(X)$, then the ED polynomial $f_{\epsilon}(y_0)$ is a polynomial in the variable $\epsilon$ having precisely ED degree many roots (including multiplicities), with at least one of them a double root. Every double root corresponds to a radius of curvature (squared) at the corresponding point $x_0$ of $X$. 
\end{remark}
\begin{definition}
Let $X$ be an algebraic variety and let $f_{\epsilon}(y)\in\mathbb{C}[y_1,\ldots,y_n][\epsilon]$ be its ED polynomial. Then for every $y_0\in \Sigma(X)$ we call all $\epsilon_0$ double roots of $f_{\epsilon}(y_0)$ \demph{radii of curvature} at corresponding points $x_0$ of $X$, where each $x_0 \in \mathrm{pr}_1(\mathrm{pr}_2^{-1}(y_0))$. Here $\mathrm{pr}_1$ and $\mathrm{pr}_2$ are as in \ref{proj}.
\end{definition}

Remember that every double root of $f_{\epsilon}(y)$ is also root of $\partial f_{\epsilon}(y_0)/\partial \epsilon.$ So for our further understanding we consider the following variety in $\mathbb{C}^n\times\mathbb{C}$ defined by the closure of
\begin{equation}\label{Def_Curv1}
\{(y,\epsilon)\in \mathbb{C}^n\times \mathbb{C}\ |\ f_{\epsilon}(y)=0, \partial f_{\epsilon}(y)/\partial \epsilon=0,\ y\in \Sigma(X)\} \setminus\{(y,\epsilon)\in \mathbb{C}^n\times \mathbb{C} \ |\ \epsilon=0\}
\end{equation}
This variety consists of pairs $(y,\epsilon)$ for which $\epsilon$ is a double root of $f_{\epsilon}(y)$ and $y\in \Sigma(X)$, so pairs of centres of curvature $y$ with corresponding non-zero radii of curvature $\epsilon$. We call this variety the \demph{curvature correspondence} of the variety $X$ and we denote it by $\mathrm{Curv}(X)$.

\begin{proposition}\label{codim2}
Generically the curvature correspondence $\mathrm{Curv}(X)$ is of codimension two in $\mathbb{C}^n\times \mathbb{C}$.
\end{proposition}

\begin{proof}
If for some $(y_0,\epsilon_0)$ we get that $f_{\epsilon_0}(y_0)=0$ and $ \partial f_{\epsilon_0}(y_0)/\partial \epsilon=0$, then it follows that $\mathrm{Res}(f_{\epsilon}(y),\partial f_{\epsilon}(y)/\partial \epsilon)(y_0)=0$, hence also $\mathrm{Disc}_{\epsilon}(y_0)=0$. And the other way around if for some $y_0$ we have that $\mathrm{Disc}_{\epsilon}(y_0)=0$, then there exists some $\epsilon_0$ (generically non-zero), such that $f_{\epsilon_0}(y_0)=0$ and $ \partial f_{\epsilon_0}(y_0)/\partial \epsilon=0$. 
So we get that projection $\mathrm{pr}$ to the first $n$ coordinates of the variety $\{(y,\epsilon)\in \mathbb{C}^n\times \mathbb{C},\text{ s.t. }f_{\epsilon}(y)=0, \partial f_{\epsilon}(y)/\partial \epsilon=0\}$ is exactly $\Delta(X)$ the offset discriminant. 

Now by the argument in~\ref{finproj} $\mathrm{pr}$ is a finite projection of degree at most the ED degree of $X$ (equal to the $\epsilon$-degree of $f_{\epsilon}(y)$, see~\cite[Theorem~$2.9$]{HW19}).
Moreover  by Proposition~\ref{Split} we have that $\Delta(X)$ is the union of two subvarieties $B(X,X)\cup \Sigma(X)$, where both $B(X,X)$ and $\Sigma(X)$ are generically $n-1$ dimensional. So the inverse image of $\Sigma(X)\subseteq \Delta(X)$ under the projection $\mathrm{pr}$ is also $n-1$ dimensional. But the inverse image of $\Sigma(X)$ under the projection $\mathrm{pr}$ is exactly $\mathrm{Curv}(X)$.
\end{proof}

\section{Critical curvature pairs}
Now we are interested in those points of the curvature correspondence, where the corresponding radius of curvature is extremal, these are points $(y,\epsilon)$ for which locally on $\mathrm{Curv}(X)\subseteq \mathbb{C}^n\times\mathbb{C}$ the $\epsilon$-direction is extremal. Suppose that the curvature correspondence of $X$ has a radical ideal $I(\mathrm{Curv}(X))$ generated by $\{g_1,\ldots,g_s\}$ and suppose that $\mathrm{Curv}(X)$ has codimension $c$ (which by Proposition~\ref{codim2} is generically equal to two).

Now we want to find all constrained critical points of the function $h(y_1,\ldots,y_n,\epsilon)=\epsilon$ with constraint $(y_1,\ldots,y_n,\epsilon)\in \mathrm{Curv}(X)$, that is with constraints $g_i(y_1,\ldots,y_n,\epsilon)=0$. We can do this by using Lagrange multipliers. This will yield that points $(y,\epsilon)$ on $\mathrm{Curv}(X)\subseteq \mathbb{C}^n\times\mathbb{C}$ for which the $\epsilon$-direction is extremal are exactly the solutions of
\[
\begin{cases}
(0,0,\ldots,0,1)+\displaystyle{\sum_{i=1}^s}\lambda_i\nabla g_i(y_1,\ldots,y_n,\epsilon)=0,\\
g_1(y_1,\ldots,y_n,\epsilon)=\ldots= g_s(y_1,\ldots,y_n,\epsilon)=0.
\end{cases}
\]
Where the upper condition is equivalent to all the $(c+1)\times (c+1)$ minors of the $(s+1)\times (n+1)$ matrix
\[\left(
\begin{array}{ccccc}
0 & 0 & \ldots &0& 1\\
\partial g_1/\partial y_1 & \partial g_1/\partial y_2& \ldots& \partial g_1/\partial y_n&\partial g_1/\partial \epsilon\\
\vdots & \vdots& \ldots& \vdots\\
\partial g_s/\partial y_1 & \partial g_s/\partial y_2& \ldots& \partial g_s/\partial y_n& \partial g_s/\partial \epsilon\\
\end{array}\right)
\] vanishing. But this boils down to all the $c\times c$ minors of the $s\times n$ matrix $\left(\partial g_i/\partial y_j\right)_{i,j}$  vanishing.

\begin{definition}\label{critcurvdeg}
The critical curvature points of the curvature correspondence are pairs $(y,\epsilon)\in\mathrm{Curv}(X)$ such that all the $c\times c$ minors of  $\left(\partial g_i/\partial y_j\right)_{i,j}$ vanish. Let us denote this variety by $\mathrm{CritCurv}(X)$ and let us call the degree of its radical ideal the \demph{critical curvature degree} of the variety $X$. 
\end{definition}

\begin{remark}
Note that there are two types of solutions to the system above. A resulting solution $(y,\epsilon)$ is either a smooth point of $\mathrm{Curv}(X)$ that corresponds indeed to a critical curvature point or any singular point of $\mathrm{Curv}(X)$ that might- or might not correspond to an actual critical curvature point of the variety.
\end{remark}

\begin{example}[Ellipse, $\mathrm{CritCurv}$ degree $4$]\label{ellipse}
Let $X$ be an ellipse in $\mathbb{C}^2$ defined by $4x_1^2+x_2^2=4$. First the algorithm below computes its ED polynomial.
	\begin{leftbar}
	\begin{verbatim}
	n=2;
	R=QQ[x_1..x_n,y_1..y_n,e];
	f=4*x_1^2+x_2^2-4;
	I=ideal(f);
	c=codim I;
	Y=matrix{{x_1..x_n}}-matrix{{y_1..y_n}};
	Jac= jacobian gens I;
	S=submatrix(Jac,{0..n-1},{0..numgens(I)-1});
	Jbar=S|transpose(Y);
	EX = I + minors(c+1,Jbar);
	SingX=I+minors(c,Jac);
	EXreg=saturate(EX,SingX);--This is the ED correspondence
	distance=Y*transpose(Y)-e;
	Offset_Correspondence=EXreg+ideal(distance);
	Off_hypersurface=eliminate(Offset_Correspondence,toList(x_1..x_n));
	G=gens Off_hypersurface;
	EDpoly=G_(0,0);
	\end{verbatim}
\end{leftbar}
The resulting output reveals that the ED polynomial is equal to
\[
f_{\epsilon}(y_1,y_2)=16y_1^8+40y_1^6y_2^2+ \ldots -360\epsilon+144,
\] a polynomial with $34$ terms, of total degree $8$ and of $\epsilon$-degree $4$ (being equal to the ED degree of the ellipse).
Now we continue by constructing its curvature correspondence $\mathrm{Curv}(X)$ by the following.
\begin{leftbar}
	\begin{verbatim}
	dif=diff(e,EDpoly);
	PreCurv=ideal(dif,EDpoly);
	Dec=decompose PreCurv
	Curv=saturate(Dec_2,ideal(e));
	\end{verbatim}
\end{leftbar}
Here $\texttt{PreCurv}$ is the variety defined by the vanishing of $f_{\epsilon}(y)$ and of $\partial f_{\epsilon}(y)/\partial \epsilon$. We decompose it and select for the curvature correspondence, $\texttt{Curv}$, the component which corresponds to the intersection with $\Sigma(X)\times \mathbb{C}$.
Or alternatively we can compute the offset discriminant and factor it, to find out the generating polynomial of $\Sigma(X)$. This goes as follows.
\begin{leftbar}
\begin{verbatim}
Disc=discriminant(EDpoly,e);
factor Disc
\end{verbatim}
\end{leftbar}
As a result we get the generating polynomials for the bisector hypersurface $B(X,X)$, which in this case is the union of the two coordinate axes, as well as $g(y)$ the generating polynomial of $\Sigma(X)$. In this case we have the following.
\begin{leftbar}
\begin{verbatim}
Curv=saturate(ideal(EDpoly, dif, g), ideal(e));
\end{verbatim}
\end{leftbar}
Now we need to construct the critical curvature points as follows.
\begin{leftbar}
	\begin{verbatim}
	JacCurv=jacobian Curv;
	M=submatrix(JacCurv,{n..2*n-1},{0..numgens(Curv)-1});
	cc=codim Curv;
	CritCurv=Curv+minors(cc,M);
	decompose radical CritCurv
	CCdeg=degree radical CritCurv
	\end{verbatim}
\end{leftbar}
As a result we get that the critical curvature degree of the ellipse is $4$ and the variety of critical curvature pairs decomposes into four real points
\[
 \{(0,-3/2,1/4), (0,3/2,1/4), (-3,0,16), (3,0,16)\},
\] representing the coordinates of the four critical centres of curvature together with the corresponding radius of curvature squared. On the figure below you can see the ellipse in red, its ED discriminant in black and the four critical osculating circles in green. 

We remark that for this variety all the points of $\mathrm{CritCurv}$ are also singular points of $\mathrm{Crit}$(X).

	\begin{figure}[h]\label{Ellipse_pic}
	\begin{center}
		\vskip -0.3cm
		\includegraphics[scale=0.35]{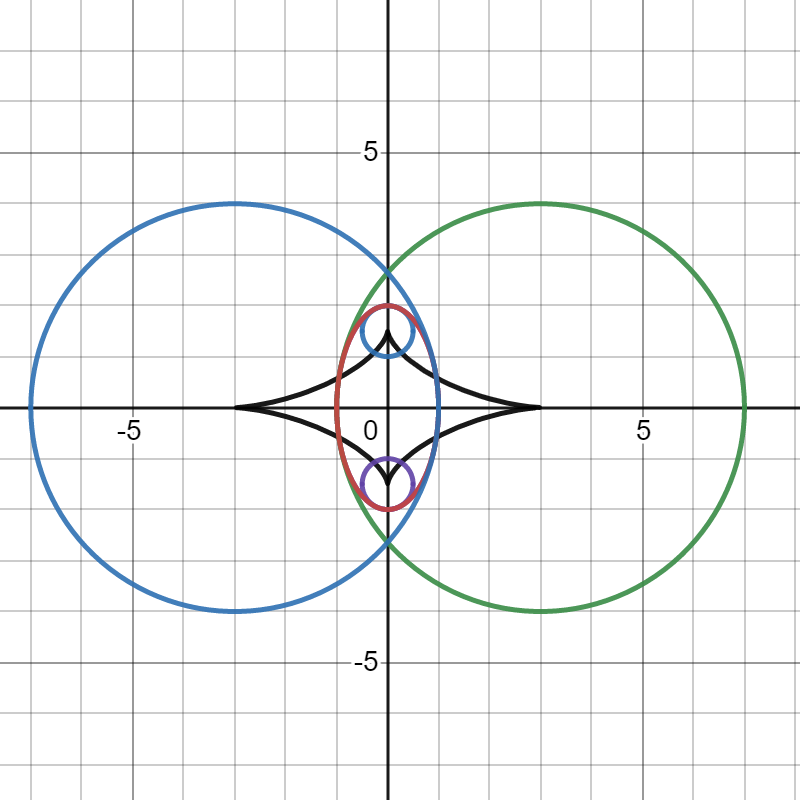}
		\caption{The evolute and four critical osculating circles of the ellipse.}
		\label{RatNorm}
	\end{center}
\end{figure}
\end{example}

\begin{remark}
	If one is interested in the points of the variety that realize these critical radii of curvature, then one could take the preimage  of the critical pairs $\mathrm{CritCurv}(X)$ under the projection $\mathrm{pr}_2:\mathcal{OC}(X)\to \mathbb{C}^n$ (see \ref{proj}). This preimage consists of pairs of points $(x,y)$, with $x\in X$ and $y$ the corresponding centre of curvature. Then taking the projection $\mathrm{pr}_1:\mathcal{OC}(X)\to \mathbb{C}^n$ to the first $n$-tuple of coordinates one gets those points of $X$ that have critical radii of curvature. So the critical curvature points of the variety $X$ are $\mathrm{pr}_1(\mathrm{pr}_2^{-1}(\mathrm{CritCurv(X)}))$.
	This is computable by the following line.
	\begin{leftbar}
		\begin{verbatim}
		decompose eliminate(CritCurv+Offset_Correspondence,toList(y_1..y_n))
		\end{verbatim}
	\end{leftbar}
\end{remark}

\begin{example}[Paraboloid, $\mathrm{CritCurv}$ degree $2$]\label{paraboloid}
In our next example we go to $\mathbb{C}^3$ and we consider the paraboloid defined by the vanishing of $f=x_3-x_1^2-x_2^2$. 
After running the computation above we find that there is a degree two imaginary curve of critical curvature points with only a single real point $(0,0,1/2,1/4)$, corresponding to the $(0,0,1/2)$ centre of the critical osculating ball of squared radius $1/4$.

We remark that for this variety as well all the points of $\mathrm{CritCurv}$ are also singular points of $\mathrm{Crit}$(X).

	\begin{figure}[h]\label{Ellipse_pic}
	\begin{center}
		\vskip -0.3cm
		\includegraphics[scale=0.4]{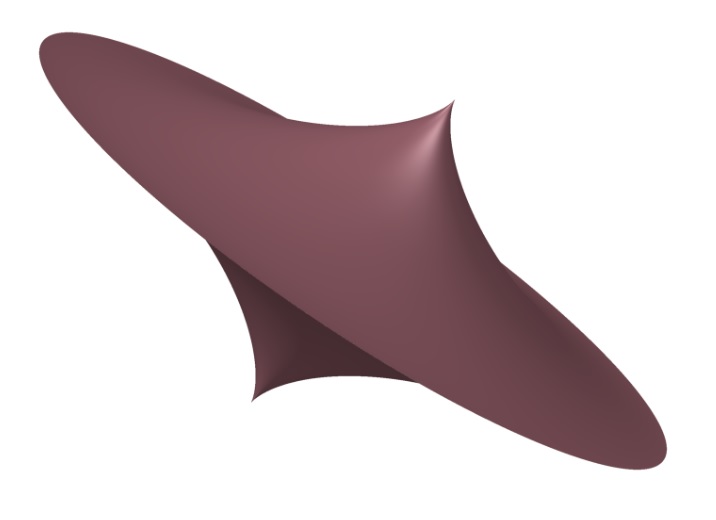}
		\vskip -0.4cm
		\caption{The ED discriminant of a symmetric ellipsoid.}
		\label{RatNorm}
	\end{center}
\end{figure}
\end{example}

\begin{example}[Ellipsoid, $\mathrm{CritCurv}$ degree $14$]\label{ellipsoid}
So far in the previous examples the real critical curvature pairs were zero dimensional regardless of the dimension of the variety. Now we consider the symmetric ellipsoid in $\mathbb{C}^3$ defined by $x_1^2+4x_2^2+4x_3^2=4$. We find that its ED discriminant is of degree $6$ and it is defined by the polynomial $64y_1^6+48y_1^4y_2^2+12y_1^2y_2^4+\ldots+243y_3^2-729$ in $16$ terms. A plot of the ED discriminant can be seen in Figure~$3$.

We find that its bisector surface is of course the  $Oy_1$ axis and the $Oy_2y_3$ coordinate plane, moreover we get that critical curvature pairs decompose as follows. Two components of degree $2$, each with a real point corresponding to one of the two cusps of the ED discriminant
\[
(-3/2,0,0,1/4)\ \text{and}\ (3/2,0,0,1/4);
\] a real circle
\[
\{(y,\epsilon)\in \mathbb{C}^3\times \mathbb{C},\text{ s.t. }\epsilon=16, y_1=0, y_2^2+y_3^2=9\},
\] corresponding to the cuspidal edge of the ED discriminant and an imaginary component 
\[\{(y,\epsilon)\in \mathbb{C}^3\times \mathbb{C},\text{ s.t. }16\epsilon^2+280\epsilon+2197=y_2^2+y_3^2+9=4y_1^2+9=0\}
\] of degree $8$. For this variety as well all the points of $\mathrm{CritCurv}$ are also singular points of $\mathrm{Crit}(X)$.
\end{example}

\begin{example}[Circle, $\mathrm{CritCurv}$ degree $2$ ]\label{circle}
So far in the previous examples the critical curvature pairs were also all singular points of $\mathrm{Curv}(X)$. Now we consider the humble circle in $\mathbb{C}^2$ defined by $x_1^2+x_2^2=1$. We find that its ED discriminant is of degree $2$ and it is defined by the polynomial $y_1^2+y_2^2=0$, so it is the isotropic quadric. We get also that critical curvature pairs are all points of the isotropic quadric with $\epsilon=1$. So $\mathrm{CritCurv}$ is of degree $2$, with one real point, that is the middle of the circle and it follows that all the points of the circle are of critical curvature.

On the other hand the singular locus of $\mathrm{CritCurv}$  only consists of the real point that is the middle of the circle with $\epsilon=1$.
\end{example}

\begin{example}[Special cuspidal cubic, $\mathrm{CritCurv}$ degree $4$]
Now we will see an example where we get a real point of $\mathrm{CritCurv}$, that is a regular point of $\mathrm{Curv}(X)$ and it corresponds to an actual critical curvature point of the variety. Our findings here are in concordance with the results in \cite{PRS21}[Section $8.2.III$]. We consider the plane curve defined by the polynomial $x_2^2-x_1^3$. We find that its ED discriminant is of degree $4$ and it is defined by the polynomial $6561y_2^4+18432y_1^3+15552y_1y_2^2+6144y_1^2+288y_2^2+512y_1$.

Moreover we find that the critical curvature pairs decompose into a real point $(0,0,4/27)$ and two imaginary components defined by
\[\{(y,\epsilon)\in \mathbb{C}^2\times \mathbb{C},\text{ s.t. }12\epsilon+1=18y_1-1=729y_2^2+64=0=0\}
\]
and by
\[\{(y,\epsilon)\in \mathbb{C}^2\times \mathbb{C},\text{ s.t. }108\epsilon+1=y_2-6y_1+1=0\}.
\]
Now interestingly the last two components are making the whole singular locus of $\mathrm{Curv}(X)$ and the one real component is actually a regular point of $\mathrm{Curv}(X)$ corresponding to a proper centre of curvature to an imaginary point on the variety. So the cuspidal cubic does not have any real critical curvature points.

On the figure below you can see the cuspidal curve in red, its ED discriminant in green.

\begin{figure}[h]\label{Ellipse_pic}
	\begin{center}
		\vskip -0.3cm
		\includegraphics[scale=0.25]{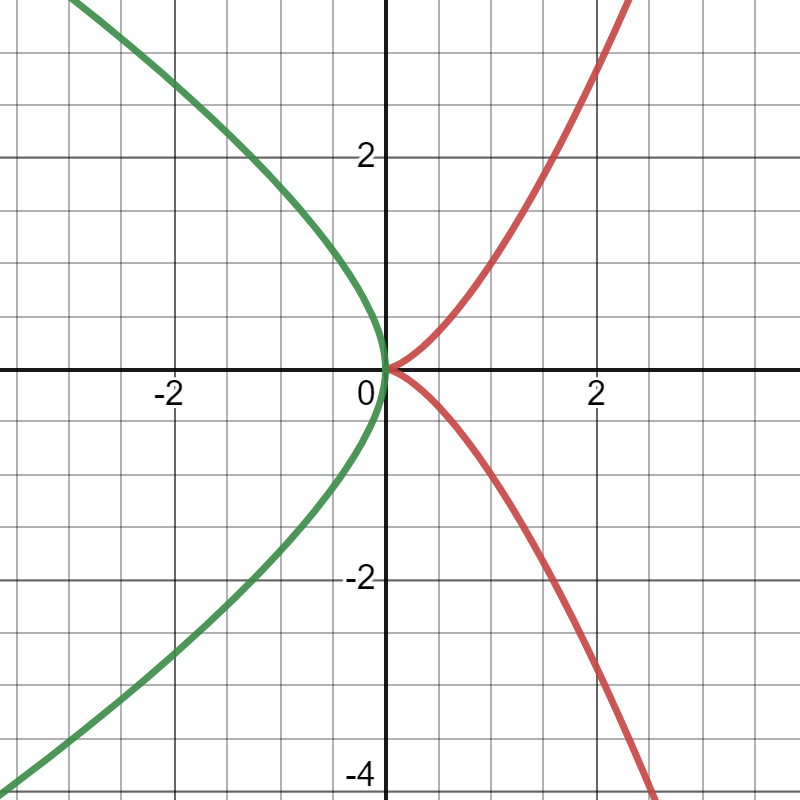}
		\caption{The cuspidal cubic and its ED discriminant.}
		\label{RatNorm}
	\end{center}
\end{figure}

\end{example}
\section{Properties of the critical curvature degree}
In this section we check some properties of the critical curvature degree of an algebraic variety. We have seen that as we defined $\mathrm{CritCurv}$ contains the singular locus of $\mathrm{Curv}$, so in order to learn about the former we have to understand the latter.

\begin{theorem}\label{Main}
Let $X\subseteq \mathbb{C}^n$ be an algebraic variety and $\mathrm{Curv}(X)$ its curvature correspondence inside $\mathbb{C}^n\times \mathbb{C}$. Let $\mathrm{pr}$ be the projection of $\mathrm{Curv}(X)$ to the first $n$ coordinates. Then we have that the preimage of the singular locus of the ED discriminant $\Sigma(X)$ under the projection $\mathrm{pr}$ is exactly the singular locus of $\mathrm{Curv}(X)$.
\end{theorem}

\begin{remark}
	For plane curves it is well known, that extremal curvature points are singular points of the evolute, generically cusps (see \cite{Li19}). Previously this result was only observed for varieties up to surfaces in $\mathbb{R}^4$ in~\cite{Port}.
\end{remark}

\begin{proof}[Proof of Theorem~\ref{Main}:]
The curvature correspondence is defined by~\ref{Def_Curv1}. Let us denote the irreducible defining polynomial of the ED discriminant $\Sigma(X)$ by $g$, then the Jacobian of the defining ideal of $\mathrm{Curv}$ (before saturation by $\epsilon=0$) is the following
\[\left(
\begin{array}{ccccc}
\partial f_{\epsilon}(y)/\partial y_1 & \partial f_{\epsilon}(y)/\partial y_2& \ldots& \partial f_{\epsilon}(y)/\partial y_n&\partial f_{\epsilon}(y)/\partial \epsilon\\
\partial^2 f_{\epsilon}(y)/\partial \epsilon\partial y_1 & \partial^2 f_{\epsilon}(y)/\partial \epsilon\partial y_2& \ldots& \partial^2 f_{\epsilon}(y)/\partial \epsilon\partial y_n&\partial^2 f_{\epsilon}(y)/\partial \epsilon^2\\
\partial g/\partial y_1 & \partial g/\partial y_2& \ldots& \partial g/\partial y_n& 0\\
\end{array}\right).
\]
First observe that every point $y_0$ on the ED discriminant is also a singular point of the corresponding offset hypersurface $\mathcal{O}_{\epsilon}$. This is true because $\mathrm{pr}_2$ (form \ref{proj}) is not of full rank over these ramification points. So for any $y_0\in\Sigma(X)$ we have that $\frac{\partial f_{\epsilon}(y)}{\partial y_i}|_{y=y_0}=0$, for all $i\in\{1,\ldots,n\}$.  We recall additionally that for any point $(y_0,\epsilon_0)\in\mathrm{Crit}(X)$ we have $\frac{\partial f_{\epsilon}(y)}{\partial \epsilon}|_{y=y_0,\epsilon=\epsilon_0}=0$ and that $\mathrm{Curv}(X)$ is of codimension $2$.

This means that the singular locus of $\mathrm{Curv}(X)$ is defined additionally by the vanishing of all the $2\times 2$ minors of 
\begin{equation}\label{jac}\left(
\begin{array}{ccccc}
\partial^2 f_{\epsilon}(y)/\partial \epsilon\partial y_1 & \partial^2 f_{\epsilon}(y)/\partial \epsilon\partial y_2& \ldots& \partial^2 f_{\epsilon}(y)/\partial \epsilon\partial y_n&\partial^2 f_{\epsilon}(y)/\partial \epsilon^2\\
\partial g/\partial y_1 & \partial g/\partial y_2& \ldots& \partial g/\partial y_n& 0\\
\end{array}\right).
\end{equation}

Now if $y_0$ is a singular point of $\Sigma(X)$, then $\frac{\partial g}{\partial y_i}|_{y=y_0}=0$, for all $i\in\{1,\ldots,n\}$ and hence all the $2\times 2$ minors of \ref{jac} are zero, so the corresponding $(y_0,\epsilon_0)$ is a singular point of $\mathrm{Curv}(X)$.

The other way around suppose by the contrary that $(y_0,\epsilon_0)$ is a singular point of $\mathrm{Curv}(X)$ such that $y_0$ is not a singular point of $\Sigma(X)$. So there exist some $i\in\{1,\ldots,n\}$, such that $\frac{\partial g}{\partial y_i}|_{y=y_0}\neq 0,$ but then we sill have that $\frac{\partial g}{\partial y_i}\cdot\frac{\partial^2 f_{\epsilon}(y)}{\partial \epsilon^2}=0$, so it follows that $\frac{\partial^2 f_{\epsilon}(y)}{\partial \epsilon^2}$ must be zero. So this means that $y_0$ is such that $f_{\epsilon}(y_0)$ has an at least triple root in $\epsilon$, but that means that $y_0$ is in the ramification locus of  $\mathrm{pr}_2$ (form \ref{proj}) restricted to $\Sigma(X)$, hence $y_0$ is a singular point of $\Sigma(X)$, leading to a contradiction. 
\end{proof}
\begin{corollary}
Let $X\subseteq \mathbb{C}^n$ be an algebraic variety and $\mathrm{CritCurv}(X)$ its critical curvature pairs variety, then we get that
\[\mathrm{deg}(\mathrm{CritCurv}(X))\geq \mathrm{deg}(\mathrm{Sing}(\Sigma(X))).\]
\end{corollary}
\begin{proof}
Observe that distinct points in the singular locus of $\Sigma(X)$ correspond to distinct points in the singular locus of $\mathrm{Curv}(X)$, so this means that the degree of the singular locus of $\mathrm{Curv}(X)$ is greater or equal to the degree of th singular locus of $\Sigma(X)$. Also remember that  $\mathrm{CritCurv}$ contains the singular locus of $\mathrm{Curv}$, so the claim follows.
\end{proof}

\begin{remark}
	For closed, smooth plane curves it is well knows that there are at least four local extrema of the curvature function (known as the Four-vertex theorem), so this case translates to $\mathrm{deg}(\mathrm{CritCurv(X))}\geq 4$. An extension of this is presented in \cite[Theorem $4.6$]{MadMad} for smooth, irreducible algebraic curves of degree $d\geq 3$, stating that $\mathrm{deg}(\mathrm{CritCurv(X)})\geq 6d^2-10d$. A further step is developed in \cite{BRW22}[Theorem $3.2$], where the authors fully characterize the critical curvature points for general quadrics in three-space.
\end{remark}

\noindent
{\bf Acknowledgements.} The author was supported by Sapientia Foundation - Institute for Scientific Research, Romania, Project No. $17/11.06.2019.$

	\vspace{1cm}
\footnotesize {\bf Authors' address:}

\smallskip

\noindent Emil Horobe\c{t}, Sapientia Hungarian University of Transylvania \ 
\hfill {\tt horobetemil@ms.sapientia.ro}

\end{document}